\documentclass[a4paper,11pt]{amsart}

\usepackage{amsfonts,latexsym,rawfonts,amsmath,amssymb,amsthm, mathrsfs, lscape}
\usepackage[all]{xy}
\usepackage[english]{babel}
\usepackage[utf8]{inputenc}
\usepackage{color, hyperref, comment}

\newcommand{\norm}[1]{\lVert#1\rVert}
\DeclareMathOperator{\tr}{tr}
\DeclareMathOperator{\Vol}{Vol}
\DeclareMathOperator{\ChernLaplacian}{\Delta_{\omega}^{Ch}}
\DeclareMathOperator{\ChernScalar}{S^{Ch}}
\DeclareMathOperator{\baseChernScalar}{\ChernScalar(\omega)}

\newtheorem{lemma}{Lemma}[section]
\newtheorem{theorem}[lemma]{Theorem}
\newtheorem{proposition}[lemma]{Proposition}
\newtheorem{conjecture}[lemma]{Conjecture}

\newtheorem{remark}[lemma]{Remark}
\newtheorem*{claim*}{Claim}

\title{A Note on Chern-Yamabe Problem}

\author{Simone Calamai}
\address[Simone Calamai]{Dipartimento di Matematica e Informatica ``Ulisse Dini''\\
Università di Firenze\\
via Morgagni 67/A\\
50134 Firenze, Italy}
\email{scalamai@math.unifi.it}
\email{simocala@gmail.com}

\author{Fangyu Zou}
\address[Fangyu Zou]{Mathematics Department\\ Stony Brook University\\ 
Stony Brook NY, 11794-3651 USA}
\email{fangyu.zou@stonybrook.edu}

\begin{document}

\begin{abstract}
We propose a flow to study the Chern-Yamabe problem and discuss the long time existence of the flow.
In the balanced case we show that the Chern-Yamabe problem is the Euler-Lagrange equation of some functional. 
The monotonicity of the functional along the flow is derived.
We also show that the functional is not bounded from below.
\end{abstract}

\maketitle

\section{Introduction}
Let $(X, \omega)$ be a compact complex manifold of complex dimension $n$ endowed with a Hermitian metric $\omega$.
The \textit{Chern scalar curvature} of $(X, \omega)$ is the scalar curvature with respect to the Chern connection associated to $\omega$.
The Chern scalar curvature can be succinctly expressed as 
\begin{equation*}\label{equ1-001} 
\baseChernScalar = \tr_\omega i\bar{\partial}\partial \log \omega^n,
\end{equation*}
where $\omega^n$ denotes the volume form. Under conformal transformation, the Chern scalar curvature changes as
\begin{equation*}\label{equ1-002}
\ChernScalar\big(\exp(2f/n)\omega\big) = \exp(-2f/n) \left( \baseChernScalar - \ChernLaplacian f \right),
\end{equation*}
where $\ChernLaplacian$ is the Chern Laplacian operator\footnote{We use the analysts' convention of Laplacian operator.} with respect to $\omega$, which is defined as
\begin{equation*}\label{equ1-003}
\ChernLaplacian f := (\omega, -dd^c f)_\omega = -2 \tr_\omega i\bar{\partial}\partial f = \Delta_d f - (d f, \theta)_\omega,
\end{equation*}
where $\theta = \theta(\omega)$ is the torsion 1-form defined by $d\omega^{n-1} = \theta \wedge \omega^{n-1}$.

In \cite{ACS} the authors proposed the Chern-Yamabe problem of finding a conformal metric in the conformal class of $\omega$ whose Chern scalar curvature is constant. More specifically, it is to find a pair $(f, \lambda) \in C^\infty(X; \mathbb{R}) \times \mathbb{R}$ solving
\begin{equation}\label{equ1-004}
-\ChernLaplacian f + \baseChernScalar = \lambda \exp(2f/n).
\end{equation}
Let $d\mu$ be the volume form of the background metric. We can normalize $f$ so that 
\begin{equation}\label{normalization}
\int_X \exp(2f/n)d\mu = 1.
\end{equation} 
Then the constant $\lambda$ is exactly the total Chern scalar curvature of the background metric
\begin{equation}\label{total-scalar}
\lambda = \int_X \baseChernScalar d\mu.
\end{equation}

The case $\lambda \leq 0$ has been solved in \cite{ACS} that there exists unique solution to \eqref{equ1-004} with normalization \eqref{normalization}. The positive case $\lambda > 0$ is still an open problem. The note \cite{Calamai} corroborates that conjecutre for projectively flat manifolds; more examples were recently found in \cite{ACS2}. This note serves as some partial efforts to close the gap.

In \cite{ACS, Lejmi} two different flows were defined  to approach the study of Hermitian metrics with constant Chern scalar curvature. Here we define a different flow, in Section $2$, which has the advantage of being monotone when the problem is known to be variational. 
The main result of the present note is
\begin{proposition}
The Chern-Yamabe flow exists as long as the maximum of Chern scalar curvature stays bounded.
\end{proposition}
Then, in Section $3$ we prove that in the balanced case the functional $\mathcal{F}$ is decreasing along the flow. Again in Section $3$ we prove that the functional $\mathcal{F}$ is not bounded from below when the complex dimension of $X$ is at least $2$.
 In section $4$ we present more properties of the flow under additional assumptions.

\medskip

\subsection*{Acknowledgments.} The authors are grateful to Professor Xiuxiong Chen for his constant support and sharing his idea on the proof of the $C^0$ estimate. 
Thanks to Haozhao Li for his interest in this project. During this  research the first author's visiting assistant professor at  University of Science and Technology, and funded by the Municipal Science and Technology Commission.

\medskip

\subsection*{Normalization}
A fundamental result by P. Gauduchon in \cite{Gauduchon} states that, on any compact complex manifold of complex dimension $n \geq 2$, every conformal class of Hermitian metrics contains a \textit{standard}, also called \textit{Gauduchon}, metric $\omega$, such that $d \omega^{n-1} = 0$. Hence, we can take the Gauduchon metric in each conformal class as the background metric. Furthermore, we can normalize the Gauduchon metric so that it has volume 1. From now on we assume the background metric $\omega$ in \eqref{equ1-004} is Gauduchon with unit volume. 

\medskip

\section{Chern-Yamabe Flow}

Let $f(x;t)$ be a family of $C^\infty$ functions on $X$ parametrized by a real parameter $t$. Let $S(x;t) = \ChernScalar(\exp(2f(x;t)/n)\omega)$. The \emph{Chern-Yamabe flow} is the flow $f(x;t)$ defined by the following flow equation:
\begin{equation}
\begin{aligned}
\frac{\partial f}{\partial t} &= \frac{n}{2}\big(\lambda - S \label{flow-equ}\big) \\
&= \frac{n}{2} \exp(-2f/n) \left(\ChernLaplacian f - \baseChernScalar + \lambda \exp(2f/n)\right) 
\end{aligned}
\end{equation}
with some initial value $f_0$ satisfying the normalization constraint 
\begin{equation}
\int_X \exp(2f_0/n)d\mu = 1. \label{flow-initial-data}
\end{equation}

Under the flow some quantities are preserved.
\begin{lemma}\label{constant-integral}
Along the flow we have
\begin{enumerate}
    \item[1.] $$\int_{X} \exp(2f/n) d\mu \equiv 1.$$
    \item[2.] $$\int_X S\exp(2f/n) d\mu \equiv \lambda.$$
\end{enumerate}
\end{lemma}

\begin{proof}
1. Let \[\phi(t) = \int_X \exp(2f/n)d\mu.\] By the initial data \eqref{flow-initial-data} and the flow equation \eqref{flow-equ}, we have $\phi(0) = 1$ and
\begin{equation*}
\begin{split}
\phi'(t)
&= \frac{2}{n}\int_X  \exp(2f/n) \frac{\partial f}{\partial t} d\mu \\
&= \int_X \left(\ChernLaplacian f - \baseChernScalar + \lambda \exp(2f/n) \right) d\mu \\
&= \lambda\big(\phi(t)-1\big). 
\end{split}
\end{equation*}
It is straightforward to show that $\phi(t) \equiv 1$.

2. It follows that
\[ \int_X S\exp(2f/n) d\mu = \int_X (\baseChernScalar - \ChernLaplacian f) d\mu \equiv \lambda. \]

\end{proof}

\subsection{Evolution of the Chern scalar curvature}
Under the Chern-Yamabe flow the Chern scalar curvature $S(x;t) = \ChernScalar(\exp(2f/n)\omega)$ evolves according to the following equation
\begin{equation}\label{equ1-021}
\frac{\partial S}{\partial t} =  \frac{n}{2} \exp(-2 f/n) \ChernLaplacian S + S (S - \lambda)
\end{equation}
with initial value $S(x;0) = \ChernScalar(\exp(2f_0/n)\omega)$.

The following lemma gives a uniform lower bound of the Chern scalar curvature.
\begin{lemma}\label{positiveness-preserving}
Let $(S_0)_{min} = \min_{x \in X} S(x;0)$. We have 
\begin{equation*}
S(x;t) \geq \min\{(S_0)_{min}, 0\}, ~\forall x \in X.
\end{equation*}
\end{lemma}

\begin{proof}
Let $S_{min}(t) = \min_{x \in X} S(x;t)$. Applying maximum principle to (\ref{equ1-021}) we obtain 
\begin{equation*}
{S_{min}}'(t) \geq S_{min}(S_{min}-\lambda)\geq -\lambda S_{min}.
\end{equation*}
Hence, 
\begin{equation*}
S(x;t) \geq S_{min}(t) \geq (S_0)_{min} \exp(-\lambda t),\forall x \in X.
\end{equation*}
If $(S_0)_{min} \geq 0$, then $S(x;t) \geq 0$; otherwise $S(x;t) \geq (S_0)_{min}$. Hence,
\[ S(x;t) \geq \min\{(S_0)_{min}, 0\}.\]
\end{proof}

\begin{remark}\label{special-initial-value}
For Lemma \ref{positiveness-preserving}, $S(x,t) \geq (S_0)_{min} \exp(-\lambda t)$ as long as the flow exists. In particular, if the initial Chern scalar curvature is strictly positive, then the positiveness is preserved along the flow. 

We can always take a special initial $f_0$ so that the initial Chern scalar curvature is strictly positive. Let $h \in C^\infty(X;\mathbb{R})$ such that \[\ChernLaplacian h = \baseChernScalar - \lambda \text{~with~}\int_X \exp(2 h / n) d\mu = 1.\] We have $\ChernScalar(\exp(2h/n)\omega) = \lambda \exp(-2h/n) > 0$. Hence, the Chern-Yamabe flow with this specific initial $f(x;0) = h(x)$ has the positive Chern scalar curvature as long as the flow exists.
\end{remark}

\subsection{Long time existence}

In this section we show that the Chern-Yamabe flow exists as long as the maximum of Chern scalar curvature stays bounded. The short time existence of the flow is straightforward as the principal symbol of the second-order operator of the right-hand side of the Chern-Yamabe flow is strictly positive definite. To obtain the long time existence, one needs to show the \emph{a priori} $C^k$ estimate 
\[ \max_{0\leq t < T} \norm{f(x;t)}_{C^k(X)} \leq C_k(T) < \infty \]
for any $T < \infty$ and any positive integer $k$. 
We use $C(T)$ to denote a constant depending on $T$. The constants $C(T)$ may vary from line to line. We begin with a $C^0$ estimate on the flow $f(x;t)$.
\begin{lemma}[$C^0$ estimate]
Suppose that the Chern-Yamabe flow exists on $\Omega_T = X \times [0, T)$ for some $T > 0$. Then these exists some constant $C_0(T)$ depending only on $(X,\omega)$ and initial data $f_0$ such that 
\begin{equation*}
\sup_{0\leq t < T}\norm{f(x;t)}_{C^0(X)} \leq C_0(T).
\end{equation*}
\end{lemma}

\begin{proof}
Let $h \in C^\infty(X; \mathbb{R})$ such that 
\[ \ChernLaplacian h = \baseChernScalar - \lambda \quad \text{with} \quad \int_X \exp(2h/n) d\mu = 1.\]
Such a function $h$ exists because of \eqref{total-scalar}. 
Similarly, by Lemma \ref{constant-integral} there exists some $v(t) \in C^\infty(X\times [0,T); \mathbb{R})$ such that 
\begin{equation}\label{intermediate-v-variable}
\ChernLaplacian v = \exp(2f/n) - 1.
\end{equation}
Differentiating \eqref{intermediate-v-variable} with respect to $t$, by the flow equation \eqref{flow-equ} we have
\begin{equation*}
\begin{aligned}
\frac{\partial}{\partial t} \left( \ChernLaplacian v \right)
&= \ChernLaplacian f - \baseChernScalar + \lambda \exp(2f/n) \\
&= \ChernLaplacian f - (\baseChernScalar - \lambda) + \lambda (\exp(2f/n) - 1) \\
&= \ChernLaplacian f - \ChernLaplacian h + \lambda \ChernLaplacian v.
\end{aligned}
\end{equation*}
Hence,
\begin{equation*}
\ChernLaplacian \left(\frac{\partial v}{\partial t} - f + h - \lambda v\right) = 0.
\end{equation*}
We can normalize $v(x;t)$ (by adding some function depending only on $t$ if necessary) so that
\begin{equation}\label{equ1-014}
\frac{\partial v}{\partial t} - f + h - \lambda v = 0 \\
\end{equation}
with initial value $v(x;0) = v_0(x)$ for some $v_0$ satisfying 
\[ \ChernLaplacian v_0 = \exp(2f_0/n) - 1 \text{~~and~~} \int_{X} v_0 d\mu = 0.\] 
Let $w(x;t) = \partial v/\partial t$. Differentiating \eqref{equ1-014} with respect to $t$, we have
\begin{equation}
\left\{
\begin{aligned}
& \frac{\partial w}{\partial t} 
= \frac{\partial f}{\partial t} + \lambda w 
= \frac{n}{2}\exp(-2f/n)\ChernLaplacian w + \lambda w, \\
& w(x;0) = f_0(x) - h(x) + \lambda v_0(x).
\end{aligned}
\right.
\end{equation}
Let $w_{max}(t) = \max_{x \in X} w(x;t)$ and $w_{min}(t) = \min_{x \in X} w(x;t)$. By maximum principle, we have
\[ \frac{d}{d t} w_{max} \leq \lambda w_{max} \text{~~and~~} \frac{d}{d t} w_{min} \geq \lambda w_{min}. \]
It follows that 
\[ w_{min}(t) \geq w_{min}(0) \exp(\lambda t) \text{~~and~~} w_{max}(t) \leq w_{max}(0)\exp(\lambda t). \]
Hence, we have 
\begin{equation*}
\norm{w(x;t)}_{C^0(X)} \leq K \exp(\lambda t) \text{~~with~~} K = \max\big(|w_{min}(0)|, |w_{max}(0)|\big).
\end{equation*}
It then follows that 
\begin{equation*}
\begin{aligned}
|v(x;t)| &=\left| v_0(x) + \int_0^t w(x;t) d t\right| \\
&\leq \norm{v_0}_{C^0(X)} + \int_0^t \norm{w(x;t)}_{C^0(X)} d t 
\leq \norm{v_0}_{C^0(X)} + \frac{K}{\lambda} \exp(\lambda t).
\end{aligned}
\end{equation*}
By \eqref{equ1-014} we have $f(x;t) = w(x;t) + h(x) - \lambda v(x;t)$. Hence,
\begin{equation*}
\begin{aligned}
\norm{f(x;t)}_{C^0(X)} &\leq \norm{w(x;t)}_{C^0(X)} + \norm{h}_{C^0(X)} + \lambda \norm{v(x;t)}_{C^0(X)} \\
&\leq K\exp(\lambda t) + \norm{h}_{C^0(X)} + \lambda\left(\norm{v_0}_{C^0(X)} + \frac{K}{\lambda} \exp(\lambda t)\right) \\
&\leq \norm{h}_{C^0(X)} + \lambda \norm{v_0}_{C^0(X)} + 2K\exp(\lambda t) := C_0(T).
\end{aligned}
\end{equation*}
Since the functions $h$, $v_0$ and $w_0$ are uniquely determined by $(X,\omega)$ and $f_0$, the constant $C_0(T)$ only depends on $(X,\omega)$ and $f_0$.
\end{proof}

\begin{lemma}
Suppose the Chern-Yamabe flow exists on $\Omega_T = X \times [0,T)$ for some $T > 0$. Moreover, suppose that 
\[ \sup_{0 \leq t < T} \norm{S}_{C^0(X)} \leq C(T) < \infty. \]
Then for any $k \in \mathbb{N}$ there exists constant $C_k(T)$ such that 
\[ \sup_{0 \leq t < T} \norm{f(x;t)}_{C^k(X)} \leq C_k(T). \]
\end{lemma}
\begin{proof}
We first get the parabolic H\"{o}lder norm bound\footnote{See the definition in Chapter IV, \cite{Lieberman}.} for $f$. For any $p \geq 1$ and $0 \leq t < T$,
\[
\begin{split}
& \norm{f(x;t)}_{W^{2,p}(X)} \leq C_p\Big(\norm{f(x;t)}_{L^p(X)} + \norm{\ChernLaplacian f(x;t)}_{L^p(X)}\Big) \\
& \leq C\Big(\sup_{0 \leq t < T} \norm{f(x;t)}_{C^0(X)} + \sup_{0 \leq t < T} \norm{S(x;t)}_{C^0(X)} + \norm{\baseChernScalar}_{C^0(X)}\Big) \\
& \leq C(T).
\end{split}
\]
By Sobolev embedding, there exists some $\alpha$ with $0 < \alpha < 1$,
\[ \sup_{0 \leq t < T} \norm{f(x;t)}_{C^{\alpha}(X)} \leq C(T). \]
Moreover, note that $|\partial f/\partial t| =(n/2) |\lambda - S| \leq C(T)$. 
Hence, we have
\[ \norm{f}_{C^\alpha(X \times [0, T))} \leq C(T). \]

Let $\mathcal{L}$ be any differential operator in $x$ and $t$. A simple calculation shows that
\[ \frac{\partial}{\partial t}(\mathcal{L} f) -  \frac{n}{2}\exp(-2f/n)\ChernLaplacian (\mathcal{L} f) + S (\mathcal{L} f) = -\frac{n}{2}\exp(2f/n)(\mathcal{L} \baseChernScalar). \]
By the interior Schauder estimate for parabolic equations (Theorem 4.9 in \cite{Lieberman}), for any $\tau, \tau'$ with $0 \leq \tau < \tau' < T$, we have
\[ \norm{\mathcal{L} f}_{C^{2+\alpha}(X \times (\tau', T))} \leq C_{Sch} (\norm{\mathcal{L} f}_{C^0(X \times (\tau, T))} + \norm{\mathcal{L}\baseChernScalar}_{C^\alpha(X \times (\tau, T))})\]
where the constant $C_{Sch}$ depends on $\tau$, $\tau'$, $\norm{S}_{C^0(X \times (\tau, T))}$ and $\norm{f}_{C^\alpha(X \times (\tau, T))}$.
It then follows by the standard bootstrapping argument to obtain that for any $\tau > 0$, $k \in \mathbb{N}$ and $0 < \alpha < 1$, there exists constant $C(k, \alpha, \tau, T)$ such that
\[ \norm{f}_{C^{k+\alpha}(X \times (\tau, T))} \leq C(k, \alpha, \tau, T).\]
Together with the short time existence near $t=0$, we have
\[ \sup_{0 \leq t < T} \norm{f(x;t)}_{C^k(X)} \leq C_k(T) < \infty. \]
\end{proof}

With Lemma \ref{positiveness-preserving} in hand, we only need $S(x;t)$ being upper bounded from infinity to obtain the $C^k$ estimate of the flow. Therefore, we have the following long time existence result.

\begin{proposition}
The Chern-Yamabe flow exists as long as the maximum of Chern scalar curvature stays bounded.
\end{proposition}

We therefore put forward the following conjecture to fully resolve the long time existence of the flow.
\begin{conjecture}
Suppose the Chern-Yamabe flow exists on $\Omega_T = X \times [0, T)$ form some $T > 0$. Then there exists some constant $C(T)$ depending on $T$ such that 
\[ S(x;t) \leq C(T), ~\forall (x,t)\in \Omega_T. \]
\end{conjecture}

\medskip

\section{Balanced Case}
A Hermitian metric on a compact complex manifold is called \textit{balanced}, if its torsion 1-form $\theta$ vanishes. A balanced metric is automatically Gauduchon, but the reverse is not necessarily true. If the background metric $\omega$ is balanced, the Chern Laplacian identifies with the Hodge-de Rham Laplacian $\ChernLaplacian = \Delta_d$. In this section we assume the background metric $\omega$ is \emph{balanced}.

\subsection{$\mathcal{F}$ functional}
When the background metric is balanced, the partial differential equation \eqref{equ1-004} with normalization \eqref{normalization} is the Euler-Lagrange equation for the following functional
\begin{equation}\label{equ1-006}
\mathcal{F}(f) := \frac{1}{2} \int_X |d f|^2_\omega d\mu + \int_X \baseChernScalar f d\mu
\end{equation}
with constraint
\begin{equation}\label{equ1-007}
\int_X \exp(2f/n) d\mu = 1.
\end{equation}
To solve the partial differential equation \eqref{equ1-004} is then equivalent to find a critical point of the functional \eqref{equ1-006} with constraint \eqref{equ1-007}. It would be nice if the functional could be bounded from below. However, this is not the case when the complex dimension $n \geq 2$.

\begin{proposition}\label{lem1-001}
For $(X,\omega)$ with complex dimension $n \geq 2$, we have
\[ \inf\left\{\mathcal{F}(f) : f \in C^{\infty}(X) \text{~with~} \int_X \exp(2f/n) d \mu = 1 \right\} = -\infty. \]
\end{proposition}

\begin{proof}
We will construct a family of Lipschitz functions $\{f_r\}$ parameterized by a positive real number $r$, each of which satisfies the constraint \eqref{equ1-007}, yet $\lim_{r \to 0} \mathcal{F}(f_r) = -\infty$. Choose an arbitrary point $p \in X$ as the center. The function $f_r(x)$ is defined as constants both inside the geodesic ball $B_r(p)$ and outside the larger ball $B_{2r}(p)$, while interpolated linearly on the annulus $B_{2r}(p)/B_{r}(p)$, namely,
\begin{equation*}
f_r (x) = \left\{
\begin{aligned}
& c_r,  & |x| \leq r  \\
& (\log r - c_r)\big({|x|}/{r} - 1\big) + c_r,  & r \leq |x| \leq 2r \\
& \log r,  & |x| \geq 2r 
\end{aligned}
\right.
\end{equation*}
where $|x|$ denotes the distance to the center of the geodesic ball and $c_r$ is a constant depending on $r$. Choose the radius $r$ sufficiently small, then the geodesic ball $B_r(0)$ is close to a Euclidean ball and $\log r < 0$. The constant $c_r$ is determined so that 
$$\int_X \exp(2f_r/n)d\mu = 1.$$ 
We claim 
$$c_r \leq -n^2 \log r - \frac{n}{2} \log C$$
for some dimensional constant $C=C(n)$. To see this,
\begin{equation*}
1 = \int_X \exp(2f_r/n) d\mu \geq \int_{B_r(p)} \exp(2c_r/n) d\mu = \exp(2c_r/n) \Vol(B_r(p)).
\end{equation*}
Hence,
$$c_r \leq -\frac{n}{2}\log \Vol(B_r(p)) = - \frac{n}{2} \log (Cr^{2n}) = -n^2\log r -\frac{n}{2}\log C.$$

Now we show that $\lim_{r \to 0} \mathcal{F}[f_r] = -\infty$. First of all, we have
\begin{equation}\label{equ1-008}
\begin{split}
\mathcal{F}(f_r) &= \int_X |d f_r|^2_\omega d\mu + \int_X \baseChernScalar f_r d\mu \\
&= \int_{B_{2r}(p) \backslash B_r(p)} |d f_r|^2_\omega d\mu + \int_{B_{2r}(p)} \baseChernScalar f_r d\mu \\
&\quad + \int_{X \backslash B_{2r}(p)} \baseChernScalar f_r d\mu_\omega.
\end{split}
\end{equation}
By continuity there exists some $r_0 > 0$ such that 
\[ \int_{B_{2r}(p)} \baseChernScalar d \mu \leq \frac{\lambda}{2}, \forall r \leq r_0. \]
Note that $\lambda = \int_X \baseChernScalar d \mu$, hence, 
\[ \int_{X \backslash B_{2r}(p)} \baseChernScalar d \mu \geq \frac{\lambda}{2},\forall r \leq r_0. \]
Take $r$ sufficiently small so that $\log r < 0$. Then $c_r > 0$ since $$\int_X \exp(2f_r/n)d\mu = 1.$$
It follows that 
\begin{equation}\label{equ1-008}
\begin{split}
\mathcal{F}(f_r)
&\leq \frac{(c_r-\log r)^2}{r^2} \Vol\big(B_{2r}(p)\backslash B_r(p)\big) \\
& \quad + \norm{\baseChernScalar}_{C^0(X)}c_r \Vol(B_{2r}(p)) + \frac{\lambda}{2} \log r \\
& \leq C(c_r - \log r)^2 r^{2n-2} + C r^{2n} c_r + \frac{\lambda}{2}\log r \\
& = \frac{\lambda}{2}\log r + O(r^{2n-2}\log r).
\end{split}
\end{equation}
When $n \geq 2$, we have $\lim_{r \to 0} r^{2n-2}\log r = 0$. The leading term for $\mathcal{F}(f_r)$ is 
$\frac{\lambda}{2}\log r$. Therefore, $\lim_{r\to 0} \mathcal{F}(f_r) = -\infty$. This finishes the proof. 
\end{proof}


\subsection{Monotonicity along the Chern-Yamabe flow}

Let $\mathcal{F}(t) = \mathcal{F}(f(\cdot;t))$. We have the following lemma showing the monotonicity of the $\mathcal{F}$ functional along the flow.

\begin{lemma}\label{monotonicity}
\begin{equation*}
\frac{d}{d t}\mathcal{F}(t) = - \int_X (S-\lambda)^2 \exp(2f/n) d\mu.
\end{equation*}
\end{lemma}

\begin{proof}
First, by Lemma \ref{constant-integral}, we have
\begin{equation*}
\int_X \frac{\partial f}{\partial t}\exp(2f/n) d\mu = 0.
\end{equation*}
Hence,
\begin{equation*}\label{derivative}
\begin{split}
\frac{d}{d t}\mathcal{F}(t) 
&= \int_X \frac{\partial f}{\partial t}(-\Delta_d f + \baseChernScalar) d\mu \\
&= \int_X \frac{\partial f}{\partial t}(-\Delta_d f + \baseChernScalar - \lambda \exp(2f/n)) d\mu \\
&= - \int_X (S-\lambda)^2 \exp(2f/n) d\mu \leq 0.
\end{split} 
\end{equation*}
The proof is finished.
\end{proof}

\subsection{Second variation}

\begin{lemma} 
The second variation of $\mathcal{F}$ functional is given by
\begin{equation}
\delta^2\mathcal{F}(u, v)\left|\right._f = \int_{X} \left((du, dv)_\omega - \frac{2\lambda}{n} \exp(2f/n) uv \right) d\mu
\end{equation}
for any $u$ and $v$ in the tangent space of $f$, namely
\[ \int_{X} \exp(2f/n) u d\mu = 0 \text{~~~and~~~} \int_X \exp(2f/n) v d\mu = 0. \]
\end{lemma}

\begin{proof}
Note that the unconstrained functional is 
\[ \tilde{\mathcal{F}}(f) = \frac{1}{2}\int_X |d f|^2_\omega d\mu + \int_X \baseChernScalar f d\mu - \frac{n\lambda}{2} \left(\int_{X} \exp(2f/n) d\mu - 1\right).\]
The second variation follows by simple calculation.
\end{proof}

Given some specific direction $v$, we have the second variation at $v$ as 
\[ \delta^2 \mathcal{F}(v) = \int_X \left(|dv|^2 - \frac{2\lambda}{n} \exp(2f/n) v^2 \right) d\mu. \]
It's interesting that the positivity of the second variation may have some relation with the Rayleigh quotient, or the first principal eigenvalue of the Laplacian operator $\lambda_1$. In the special case when the background Gauduchon metric is itself a constant Chern-Scalar curvature metric, we have $f = 0$ is a critical point. 

If $\lambda_1 \geq 2\lambda/n$, then
\[ \delta^2 \mathcal{F}(v) \geq (\lambda_1 - 2\lambda/n) \int_X v^2 d\mu \geq 0,~~\forall v \text{~~with~~} \int_X v d\mu = 0 \]
shows that $f = 0$ is a local minimum. 

If $\lambda_1 < 2\lambda/n$, then we can take some non-zero eigenvector $v_0$ with $\int_X v_0 d\mu = 0$ and
\[ \delta^2 \mathcal{F}(v_0) \leq (\lambda_1 - 2\lambda/n) \int_X v_0^2 d\mu < 0. \]
Hence, $f = 0$ is a saddle point and unstable. 

To construct concrete example for the above argument, 
one can consider $\mathbb{P}^1 \times \theta \mathbb{P}^1$ with $\mathbb{P}^1$ and $\mathbb{P}^1$ both endowed with the standard Fubini-Study metrics. For such family of complex manifolds, the background Fubini Study metrics $\omega_\theta$ are constant Chern scalar curvature metrics; so we write down the functional $\mathcal{F}$ with respect to the reference metric $\omega_\theta$, and $f=0$ 
represents a constant scalar Chern curvature metric with $\mathcal{F}(0)= 0$.
By adjusting the scaling parameter $\theta$, it is not hard to adjust $\lambda_1$ and the total Chern scalar curvature $\lambda$ such that $-\frac{\lambda_1}{2}+\lambda <0$; this makes possible to find a sequence of 
conformal factors $f_k$ that are arbitrarily close to $f=0$, and with $\mathcal{F}(f_k) < 0$. Since the flow decreases the functional $\mathcal{F}$, then the flow starting at $f_k$ will not converge to $f=0$.
The conclusion we can draw is that saddle points are possible and we should not expect only local minima in general. Together with the fact, proved in Lemma \ref{monotonicity}, that the $\mathcal{F}$ functional always is not bounded from below, the techniques for only minima is not enough. 

\section{Additional results under assumptions}
We have already shown in Lemma \ref{lem1-001} that the functional $\mathcal{F}$ is unbounded from below. So it is impossible to find a global minimum. Yet it is still possible that the functional is bounded from below along the flow for some specific initial value. In particular, if the flow finally converges to a solution, one of the necessary conditions is that the functional is bounded under the flow.

In this section we assume the flow exists on $[0,\infty)$ and 
\begin{equation}\label{lower-bound-assumption}
\lim_{t\to\infty} \mathcal{F}(t) \geq C > -\infty.     
\end{equation}
What can we say about the flow?

Since the functional is decreasing and bounded from below, we can find a sequence of time slices $\{t_k\}$, so that $\frac{d}{d t}\mathcal{F}(t_k) \to 0$.
Let $f_k = f(t_k)$ and $S_k = S\big(\exp(2f_k/n)\omega\big)$. 
Note that by Lemma \ref{monotonicity}, 
$$\frac{d}{d t}{\mathcal{F}}(t) = -\int_X (S-\lambda)^2 \exp(2f/n)d\mu = \lambda^2 - \int_X S^2\exp(2f/n)d\mu.$$
On the other hand, by Lemma \ref{positiveness-preserving}, we have $S(x;t) > -C$.
Hence, we have
\begin{equation}\label{conditions}
\left\{
\begin{aligned}
& \int_X \exp(2f_k/n) d\mu = 1, \\
& \int_X S_k^2 \exp(2f_k/n) d\mu \to \lambda^2 \text{~and~} S_k > -C, \\
& |\mathcal{F}(f_k)| \leq C.
\end{aligned}
\right.
\end{equation}

\subsection{Assuming uniform upper bound}
In this section we assume that there exists uniform upper bound for the sequence $\{f_k\}$ in \eqref{conditions}. We show that there exists a smooth solution to the Chern-Yamabe equation \ref{equ1-004}. In what follows the constant $C$ may vary from line to line.

\begin{lemma}\label{H2-estimate-under-assumption}
Suppose there is a sequence $\{f_k\}$ satisfying \eqref{conditions}. Suppose additionally there exists some constant $C_0$ such that \[\max_{x \in X} f_k(x) \leq C_0,  \forall k.\] Then
$\norm{f_k}_{H^2} \leq C.$
\end{lemma}

\begin{proof}
First of all, we have
\[\int_X (S_k\exp(2f_k/n))^2 d\mu \leq \exp(2C_0/n)\int_X S_k^2\exp(2f_k/n)d\mu \leq C.\]
Note that $S_k = \exp(-2f_k/n)(\baseChernScalar - \Delta f_k)$, namely, $\Delta f_k = \baseChernScalar - S_k \exp(2f_k/n)$. Hence, we have $\norm{\Delta f_k}_{L^2(X)} \leq C$.
\vspace{0.5em}

\begin{claim*}
Let $\bar{f_k} = \int_{X} f_k d \mu$. There exists some constant $C > 0$ such that $-C \leq \bar{f_k} \leq 0$.
\end{claim*}
\begin{proof}[Proof of the Claim]
First of all, since $\Vol(X) = 1$, we have 
\[ \exp(2\bar{f_k}/n) = \exp\Big(\int_X (2f_k/n) d\mu\Big) \leq \int_X \exp(2f_k/n) d\mu = 1.\]
Hence, $\bar{f_k} \leq 0$. 

For the other side, first note that
\begin{equation}\label{claim-1}
\begin{split}
& \int_X S_k \exp(2f_k/n) f_k d \mu \\
&= \int_X (-\Delta f_k + \baseChernScalar) f_k d \mu 
= 2 \mathcal{F}(f_k) - \int_X (\baseChernScalar - \lambda) f_k - \lambda \bar{f_k} \\
&= 2 \mathcal{F}(f_k) - \int_X \Delta h f_k - \lambda \bar{f_k} 
= 2 \mathcal{F}(f_k) - \int_X h \Delta f_k - \lambda \bar{f_k} \\
&\geq 2 \mathcal{F}(f_k) - \norm{h}_{L^2(X)}\norm{\Delta f_k}_{L^2(X)} - \lambda \bar{f_k} \\
&\geq C - \lambda \bar{f_k}.
\end{split}
\end{equation}
On the other hand, since $S_k > -C$, we have
\begin{equation}\label{claim-2}
\begin{split}
& \int_X S_k \exp(2f_k/n) f_k d\mu \\
&= \int_X (S_k + C) \exp(2f_k/n) f_k d\mu - C \int_X \exp(2f_k/n) f_k d\mu \\
&\leq C_0\exp(C_0/n) \int_X S_k \exp(f_k/n) d\mu + C_0 C+ {C}\cdot{\frac{n}{2e}} \\
& \leq C_0 \exp(C_0/n) \Big( \int_X S_k^2 \exp(2 f_k/n)d\mu \Big)^{1/2} + C
\leq C.
\end{split}
\end{equation}
By \eqref{claim-1} and \eqref{claim-2}, we obtain that $\bar{f_k} \geq -C$. This finishes the proof of the Claim.
\end{proof}

We continue our proof for the Lemma. By Poincare inequality, there exists some constant $C_p$ so that
\begin{equation*} 
\int_X (f_k - \bar{f_k})^2 d\mu \leq C_p\int_X |\nabla f_k|^2 d\mu. 
\end{equation*}
On the other hand, 
\begin{equation*} 
\int_X |\nabla f_k|^2 d\mu = \int_X (-f_k \Delta f_k) d\mu 
\leq \frac{1}{2 C_p} \int_X f_k^2 d\mu + \frac{C_p}{2} \int_X (\Delta f_k)^2 d\mu. 
\end{equation*}
Hence,
\begin{equation*} 
\int_X f_k^2 d\mu - {\bar{f_k}}^2 \leq \frac{1}{2} \int_X f_k^2 d\mu + \frac{C_p^2}{2}\int_X (\Delta f_k)^2 d\mu. 
\end{equation*}
Hence,
\begin{equation}\label{l2-estimate}
\int_X f_k^2 d\mu \leq 2 \bar{f_k}^2 + C_p^2 \int_X (\Delta f_k)^2 d\mu. 
\end{equation}
It then follows by Sobolev estimate that 
\[ \norm{f_k}_{H^2(X)} \leq C\Big(\norm{f_k}_{L^2(X)} + \norm{\Delta f_k}_{L^2(X)}\Big) \leq C\Big(|\bar{f_k}| + \norm{\Delta f_k}_{L^2(X)}\Big) \leq C. \]
This finishes the proof.
\end{proof}

\begin{proposition}
Suppose there is a sequence $\{f_k\}$ satisfying \eqref{conditions}. Suppose additionally there exists some constant $C_0$ such that \[\max_{x \in X} f_k(x) \leq C_0,  \forall k.\] Then there exists a function ${f_\infty} \in C^\infty(X)$ which solves the differential equation \eqref{equ1-004}.
\end{proposition}

\begin{proof}
By Lemma \ref{H2-estimate-under-assumption}, we have $\norm{f_k}_{H^2(X)} \leq C$. Hence, by passing to a subsequence if necessary, we have
$f_k \rightharpoonup f_\infty$ weakly in $H^2(X)$ for some $f_\infty$. It follows that $f_k \rightarrow f_\infty$ strongly in $L^2(X)$ and $\Delta f_k \rightharpoonup \Delta f_\infty$ weakly in $L^2(X)$. As a result of the strong convergence in $L^2(X)$, by passing to a subsequence if necessary, we have 
$f_k \rightarrow f_\infty$ $d\mu$-a.e.. Then by Egonov's theorem, for any $\delta > 0$, there exists a subset $\Omega_\delta \subset X$ with $\Vol(X \backslash \Omega_\delta) < \delta$, such that $f_k \rightarrow f_\infty$ uniformly on $\Omega_\delta$. We have
\begin{equation*}
\begin{split}
&\int_{\Omega_\delta} (\Delta f_k - \baseChernScalar)^2 \exp(-2f_k/n) d\mu \\
&= \int_{\Omega_\delta} (\Delta f_k - \baseChernScalar)^2 \exp(-2f_\infty/n) d\mu \\
&\quad + \int_{\Omega_\delta} (\Delta f_k - \baseChernScalar)^2 \big( e^{-2f_k/n} - e^{-2f_\infty /n} \big) d\mu \\
&\geq \int_{\Omega_\delta} (\Delta f_k - \baseChernScalar)^2 \exp(-2f_\infty/n) d\mu 
- C\norm{e^{-2f_k/n} - e^{-2f_\infty /n}}_{L^\infty(\Omega_\delta)}.
\end{split}
\end{equation*}
Hence,
\begin{equation*}
\begin{split}
&\liminf_{k \to \infty} \int_{\Omega_\delta} (\Delta f_k - \baseChernScalar)^2 \exp(-2f_k/n) d\mu \\
&\geq \liminf_{k \to \infty} \int_{\Omega_\delta} (\Delta f_k - \baseChernScalar)^2 \exp(-2f_\infty/n) d\mu \\
&\geq \int_{\Omega_\delta} (\Delta f_\infty - \baseChernScalar)^2 \exp(-2f_\infty/n) d\mu.
\end{split}
\end{equation*}
Notice that
\begin{equation*}
\int_{X} (\Delta f_k - \baseChernScalar)^2 \exp(-2f_k/n) d\mu
= \int_X S_k^2 \exp(2f_k/n) d\mu \rightarrow \lambda^2, \text{~~as~} k \rightarrow \infty.
\end{equation*}
Hence,
\begin{equation*}
\int_{\Omega_\delta} (\Delta f_\infty - \baseChernScalar)^2 \exp(-2f_\infty/n) d\mu \leq \lambda^2.
\end{equation*}
Let $\delta \rightarrow 0$, we obtain that 
\begin{equation}\label{limit-integral-1}
\int_X (\Delta f_\infty - \baseChernScalar)^2 \exp(-2f_\infty/n) d\mu \leq \lambda^2.
\end{equation}
Note that $f_k \rightarrow f_\infty$ $d\mu$-a.e., and $f_k \leq C_0$ by assumption, we have $f_\infty \leq C_0$ $d\mu$-a.e.. Then by Dominance Convergence Theorem, we have
\begin{equation}\label{limit-integral-2}
\int_X \exp(2 f_\infty/n) d\mu = \lim_{k \to \infty} \int_X \exp(2 f_k/n) d\mu = 1.
\end{equation}
By \eqref{limit-integral-1} and \eqref{limit-integral-2}, we have
\begin{equation*}
\int_X \Big(\Delta f_\infty - \baseChernScalar + \lambda \exp(2 f_\infty / n)\Big)^2 \exp(-2f_\infty/n) d\mu \leq 0.
\end{equation*}
It follows that the equality holds and 
\begin{equation}
\Delta f_\infty - \baseChernScalar + \lambda \exp(2 f_\infty/n) = 0, ~d\mu-\text{a.e.}.
\end{equation}
Since $f_\infty \leq C_0$ $d\mu$-a.e., we have $\Delta f_\infty = \baseChernScalar - \lambda \exp(2 f_\infty/n) \in L^\infty(X)$. Hence, $f_\infty \in W^{2, p}(X)$ for any $p > 1$. By Sobolev embedding theorem, this implies that $f_\infty \in C^{1,\alpha}(X)$. Then $\Delta f_\infty \in C^{1,\alpha}(X)$. By the standard bootstrapping argument, we eventually have $f_\infty \in C^\infty(X)$. This finishes the proof.
\end{proof}


\end{document}